\documentclass[12pt]{article}
\usepackage{amssymb,amsmath,latexsym}
\usepackage{amsthm}

\newtheorem{theorem}{Theorem}
\newtheorem{definition}{Definition}
\newtheorem{lemma}{Lemma}

\DeclareMathOperator{\ind}{ind}
\DeclareMathOperator{\order}{order}

\begin{document}
\title{Geometric representation of 3-restricted min-wise independent sets of permutations}
\author{Victor Bargachev}
\date{\today}
\maketitle

\abstract{
	We study properties of 3-restricted min-wise independent sets of permutations and prove 
	that the class of these sets is isomorphic to the class of subsets of vertices of multi-dimentional hypercube 
	of certain size
	with certain pair-wise distances fixed.
}

\newcommand{\m}[1]{{m\over {#1}}}

\section{Definitions and Main Theorem}

	Throughout the paper we use the following notation. Let $[n]=\{0,\dots,n-1\}$ and $S_n$ denote symmertic 
group acting on $[n]$.
	Let $\ind()$ denote truth indicator, i.e. $\ind(0<1)=1$ and $\ind(0>1)=0$. 
	Let $B=\{0,1\}$ and $B^n$ is a set of vertices of $n$-dimensional hypercube, or similary set 
of binary vectors of length $n$. For $\pi\in S_n$ we write it down as $\underline{\pi(0)\pi(1)\dots\pi(n-1)}$,
 i.e. $\pi=\underline{120}$ means that $\pi(0)=1$, $\pi(1)=2$ and $\pi(2)=0$.
	Let $\#$ denote a size of set, e.g. $\#[n]=n$.

\begin{definition}{\cite{broder}}
	A set of permutations $S\subset S_n$ is called 3-restricted min-wise-independent (3-mwi) if 
	$\forall X\subset [n]: |X|\le 3$ and $\forall x\in X$
		$$
			 \#\{\pi\in S: \pi(x)=\min(\pi(X))\} = {|S|\over |X|}.
		$$
\end{definition}

It is often more handy to use the equivalent definition.

\begin{lemma}\label{lemmaXYZ}
	A set of permutations $S\subset S_n$ is 3-mwi if and only if for all distinct $x,y,z\in [n]$
	$$
		\#\{\pi\in S: \pi(x)<\pi(y)<\pi(z)\} = {|S|\over 6}.
	$$
\end{lemma}

\begin{proof}
	Suppose $S$ is 3-mwi, let $|S|=m$ and let $f(u,v,w) := \#\{\pi\in S: \pi(u)<\pi(v)<\pi(w)\}$. Then 
		$\m2 = \#\{\pi\in S: \pi(y)<\pi(z)\} = f(x,y,z) + f(y,x,z) + f(y,z,x)$, 
		$\m3 = \#\{\pi\in S: \pi(y)<\pi(x) \wedge \pi(y)<\pi(z)\} = f(y,x,z) + f(y,z,x)$
	and hence
		$f(x,y,z) = \m2-\m3=\m6$. The back implication is obvious.
\end{proof}

\newcommand{\vxy}{v_{xy}}
\newcommand{\vyx}{v_{yx}}
\newcommand{\vxz}{v_{xz}}
\newcommand{\vyz}{v_{yz}}

Throughout the paper $S$ will denote a set of permutations of order $n$ and $m$ will denote its size: 
	$S=\{\pi_0,\dots,\pi_{m-1}\}\subset S_n$.
Now let $x,y\in [n]$ be distinct and let us define a binary vector $v_{xy}\in B^m$:
	$$
		(\vxy)_i := \ind(\pi_i(x) < \pi_i(y)).
	$$

If $S$ is 3-mwi then from definition is follows that 
	$$
		\vxy^2 = \sum_{i\in[m]} \ind(\pi_i(x) < \pi_i(y)) = \#\{i\in[m]: \pi_i(x) = \min \pi_i(\{x, y\}) \} = {m\over 2}.
	$$
Similarly, for all distinct $x,y,z\in [n]$ we have
	$$
		\vxy\cdot \vxz = \m3 \text{, } 
		\vxy\cdot \vyz = \m6 \text{ and }
		\vxy\cdot \vyx = 0.
	$$

\begin{definition}
	Let $I=\{(x,y): x,y\in[n], x \neq y\}$ be a sequence of pairs. 
	We say that set of vectors $V=\{\vxy\in B^m: (x,y)\in I\}$ indexed with $I$ is 3-mwi if for all distinct $x, y, z\in [n]$
	$$
		v_{xy}^2 = \m2,\ 
		v_{xy}\cdot v_{yx} = 0,\ 
		v_{xy}\cdot v_{xz} = \m3,\ 
		v_{xy}\cdot v_{yz} = \m6.
	$$
	For $i\in [m]$ and $x\in [n]$ let us define
	\begin{equation}
		\label{orderDefinition}
		\order_V(i, x) := 
			\#\{ y : y \neq x, (v_{xy})_i = 0 \}
		= n-1-\sum_{y\neq x} (v_{xy})_i.
	\end{equation}
\end{definition}

Above we have already shown how to construct 3-mwi set of vectors $V$ from 3-mwi set of permutations $S$. Moreover,
from a construction it can be seen that $\forall i,x\ \pi_i(x) = \order_V(i, x)$ and hence $S$ can be restored from $V$. 
What we are going to prove is that every 3-mwi set of vectors corresponds to some 3-mwi set of permutaions:

\begin{theorem}\label{mainTheorem}
	Let $V\subset B^m$ be a 3-mwi set of vectors. 
	Then $\forall i\in[m]$ and 
		$\forall x,y\in [n], x\ne y$ we have $\order_V(i,x)\ne \order_V(i,y)$ and
	\begin{equation}
		\label{piDefinition}
		S=\{\pi_i: i\in[m]\}\subset S_n \text{ where } \pi_i(x) := \order_V(i,x)
	\end{equation}
	is 3-mwi set of permutations.
\end{theorem}


\section{Proof of the Theorem \ref{mainTheorem}}

\begin{lemma}\label{lemmaUVW}
	Let $u, v, w \in B^m$ with $u^2=v^2=w^2=\m2, u\cdot v =v\cdot w=\m3$ 
and $u\cdot w=\m6$. Then 
	$$
		\forall i\ u_i=1\wedge w_i=1 \Rightarrow v_i=1.
	$$
\end{lemma}

\begin{proof}
	For $k, l, m \in B$, define 
	$c_{klm} := \# \{ i\in[m]: u_i=k, v_i=l, w_i=m \}$.
	Let $c_{kl} := c_{kl0} + c_{kl1}$. Then $v^2=\m2$ implies $c_{01}+c_{11}=\m2$, 
	but $u\cdot v=c_{11}=\m3$ and hence $c_{01}=\m2-\m3=\m6$. 
	Now by subtracting $u\cdot w=c_{101}+c_{111}=\m6$ from $v\cdot w=c_{011}+c_{111}=\m3$ obtain
	$c_{011}-c_{101}=\m6$ or $c_{101}=c_{011}-\m6 = c_{01}-c_{010}-\m6=-c_{010}$, 
	which means that $c_{101}=0$.
\end{proof}

Let set of vectors $V=\{\vxy\}$ be fixed. For a given $i$ define binary relation $\prec_i$ on $[n]$:
	$$
		x \prec_i y := (x \neq y) \wedge ((\vxy)_i = 1).
	$$


\begin{lemma} 
	Let $V$ be a 3-mwi set of vectors and $i$ be fixed. Then relation $\prec_i$ defines a liner order on $[n]$.
\end{lemma}

\begin{proof}
	Let $x, y\in [n]$ be distinct. Now $v_{xy}\cdot v_{yx}=0$ together with $\vxy^2=\vyx^2=\m2$ implies $(\vxy)_i=1-(\vyx)_i$,
	which gives $x \prec_i y \iff \neg (y \prec_i x)$, i.e. relation $\prec_i$ is anti-symmetric.

	Let $x, y, z\in [n]$ be distinct and let $x \prec_i y$ and $y \prec_i z$.   
	Put $u=\vxy, v=\vxz$, and  $w=\vyz$. By assumption $V$ is 3-mwi which implies
		$u^2=v^2=w^2=\m2$, 
		$u\cdot v=v\cdot w =\m3$, and
		$u\cdot w=\m6$.
	By definition, $x \prec_i y$ implies $u_i=1$, $y \prec_i z$ implies $w_i=1$
	and hence Lemma \ref{lemmaUVW} states that $v_i=1$, i.e. $x \prec_i z$. This means that $\prec_i$ is transitive.
\end{proof}


\begin{lemma}\label{lemma4}
	Let $V$ be a 3-mwi set of vectors and $i\in[m]$ be fixed. Then
	$$
		x\prec_i y \iff \order_V(i, x) < \order_V(i,y).
	$$
\end{lemma}

\begin{proof}
	Define $f(u) := \{v\in [n]: v\prec_i u\}$ and rewrite (\ref{orderDefinition}): $\order_V(i, u) = \#f(u)$. 
	Now $x\prec_i y$ implies 
		$\forall z (z\prec_i x \Rightarrow z\prec_i y)$ 
	and hence
		$f(x)\subset f(y)$. 
	So
		$\order_V(i, y)=\#f(y)=\#f(x)+\#(f(y)\setminus f(x)) > \order_V(i, x)$
	since
		$x\in f(y)\setminus f(x)$.
	This provides left-to-right implication and another one is symmetric.
\end{proof}


	Now we are ready to prove Theorem \ref{mainTheorem}. 
	Let $V$ be a 3-mwi vector set. 
	First of all, Lemma \ref{lemma4} states that functions $\pi_i$ defined in (\ref{piDefinition}) are indeed permutations.
		For $u,v\in [n]$ and $i\in[m]$ define
	$
		f_i(u, v) := 
			\{i: \pi_i(u) < \pi_i(v)\}
			= \{i: \order_V(i, u) < \order_V(i, v)\}
			= \{i: x\prec_i y\} = \{i: (v_{uv})_i = 1\}.
	$
	Now let $x,y,z\in[n]$ be distinct:
	$
		\#\{i : \pi_i(x) < \pi_i(y) < \pi_i(z)\} 
		= \#(f_i(x,y)\cap f_i(y,z)\}
		= \#\{i:  (\vxy)_i = 1\wedge (\vyz)_i = 1\} 
		= \vxy\cdot\vyz = \m6,
	$
	 so Lemma \ref{lemmaXYZ} provides that $S=\{\pi_i\}$ is indeed 3-mwi permutations set.


\section{Usage Example}

We will use the Theorem \ref{mainTheorem} to prove that the maximal order of 3-mwi permuations set of size 6 is 4 ($m=6\Rightarrow n\leq 4$). 

Suppose $n=5$, $S\subset S_n$ is 3-mwi and $V$ is a corresponding vector set. Without loss of generality we may
assume that
	$$	
		v_{01} = 000111,\ 
		v_{02} = 001011,\
		v_{12} = 011001.
	$$
Let $f(v,d):=\{u\in B^m: u^2=\m2, u\cdot v=d\}$. Then
	$$
		v_{03},v_{04} \in V_0 := f(v_{01},\m3)\cap f(v_{02},\m3)  = \{010011, 100011, 001101, 001110\}.
	$$
	$$
		v_{13},v_{14} \in V_1 := f(v_{01},\m6)\cap f(v_{12},\m3)  = \{101001, 110001, 011010, 011100\}.
	$$
	$$
		v_{23},v_{24} \in V_2 := f(v_{02},\m6)\cap f(v_{12},\m6)  = \{100101, 101100, 010110, 110010\}.
	$$

Let $v_{03} = 010011$. Then 
	$v_{13}\in V_1\cap f(v_{03},\m3)=\{110001, 011010\}$
and
	$v_{23}\in V_2\cap f(v_{03},\m3)=\{010110, 110010\}$.
Besides 
	$v_{04}\in V_0\cap f(v_{03},\m3)$ 
		implies 
	$v_{04}=100011$, 
	$v_{14}\in V_1\cap f(v_{04},\m3)=\{101001, 110001\}$
and
	$v_{24}\in V_2\cap f(v_{04},\m3)=\{100101, 110010\}$.
Now $v_{13}\cdot v_{14}=\m3$ implies $v_{14}=101001$, which implies $v_{24}=100101$ and $v_{13}=110001$, which 
implies $v_{23}=110010$. But $v_{23}\cdot v_{24}\neq\m3$ which contradicts the assumption.

	Similarly $v_{03} = 001101$ implies 
		$v_{04} = 001110$,
		$v_{13} = \{101001, 011100\}$, 
		$v_{23} = \{100101, 101100\}$, 
		$v_{14} = \{011010, 011100\}$ and 
		$v_{24} = \{101100, 010110\}$.
	Now 
		$v_{13}\cdot v_{14}=\m3$ 
	implies 
		$v_{13} = 011100$, $v_{14}=011010$, $v_{23}=101100$, $v_{24}=010110$
	and 
		$v_{23}\cdot v_{24}\neq \m6$.

Cases $v_{03} = 100011$ and $v_{03} = 001110$  imply $v_{04} = 010011$ and $v_{04} = 001101$ correspondingly, 
so the above deduction can be applied with $3$ and $4$ interchanged. This completes the proof for $n=5$. 

For $n=4$ let us take $v_{03} = 010011\in V_0$, $v_{13} = 110001\in V_1 $, $v_{23} = 110010\in V_2$, and check
	that 
		$v_{03}\cdot v_{13}=v_{03}\cdot v_{23}=v_{13}\cdot v_{23}=\m3$,
	which is suffient for a resulting vectors set to be 3-mwi. 
	Let us use (\ref{orderDefinition}) to restore the permuations set: 
		$\pi_0(0) = n-1 - (v_{01})_0 - (v_{02})_0 - (v_{03})_0 = 3$,
		$\pi_0(1) = n-1 - (1-v_{01})_0 - (v_{12})_0 - (v_{13})_0 = 1$,
		$\pi_0(2) = n-1 - (1-v_{02})_0 - (1-v_{12})_0 - (v_{23})_0 = 0$ and
		$\pi_0(3) = n-1 - (1-v_{03})_0 - (1-v_{13})_0 - (1-v_{23})_0 = 2$,
	so
		$\pi_0 = \underline{3102}$. Continuing doing that obtain 3-mwi permutations set 
	$S=\{\underline{3102}, \underline{2013}, \underline{2130}, \underline{2310}, \underline{0312}, \underline{0132}\}$.

\end{document}